\documentclass[a4paper,12pt]{article}
\newcounter{minutes}
\divide\time by 60
\newcounter{hours}
\multiply\time by 60
\addtocounter{minutes}{-\time}

\usepackage[english]{babel}
\usepackage[T1]{fontenc}
\usepackage{amsthm,enumerate,amsmath,amscd,amssymb}
\usepackage[dvips]{epsfig}
\usepackage[dvips]{graphicx}

\usepackage{geometry}
\font\ff=eusm10 scaled 1200
\def\K{\hbox{\ff K}}

\newcommand{\comment}[1]{}

\swapnumbers

\theoremstyle{definition}

\theoremstyle{plain}

\newtheorem{theorem}[equation]{Theorem}
\newtheorem{corollary}[equation]{Corollary}

\newtheorem{lemma}[equation]{Lemma}

\newtheorem{problem}[equation]{Problem}
\newtheorem{rmk}[equation]{Remark}

\newenvironment{subsec}
{
 \addtocounter{equation}{1}
 \bigskip
 \noindent
   {\bf \arabic{section}.\arabic{equation}.}
\begin{textbf}
} {\end{textbf}}

\numberwithin{equation}{section}

\begin{document}

\begin{center}
{\bf \large On quasiconformal maps with identity boundary values}
\end{center}

\begin{center}

{\bf V. Manojlovi\'c and M. Vuorinen}
\end{center}

\begin{center}
\texttt{File:~\jobname .tex,
          printed: \number\year-\number\month-\number\day,
          \thehours.\ifnum\theminutes<10{0}\fi\theminutes}
\end{center}

Abstract:  Quasiconformal
homeomorphisms of the unit ball $B^n$ of ${\mathbb R}^n, n \ge 3,$  onto itself with identity
boundary values are studied. A spatial analogue of Teichm\"uller's theorem is proved.


%%%%%%%%%%%%%%%%%%%%%
\section{Introduction}

For a domain $G\subset\mathbb R^n$, $n\geqslant 2$, let
$$
Id(\partial G)=\{f\,:\,\overline{\mathbb R^n} \to \overline{\mathbb
R^n} \mbox{ homeomorphism }:\,f(x)=x,\quad\forall x\in\overline{\mathbb
R^n}\setminus G\}.
$$
Here $\overline{\mathbb R^n}$ stands for the M\"obius space $\mathbb
R^n \cup \{ \infty \} \,.$ We shall always assume that $card  \{
\overline{\mathbb R^n} \setminus G \} \ge 3.$
 If $K\geqslant 1$, then the class of $K$-quasiconformal maps in
$Id(\partial G)$ is denoted by $Id_K(\partial G)$. Throughout this
paper we adopt the standard notation and terminology from V\"ais\"al\"a's book
\cite{v2}. In particular, $K$-quasiconformal maps are defined in
terms of the maximal dilatation as in \cite[p. 42]{v2} if not
otherwise stated. The maximal dilatation of a homeomorphism $f:G \to G'$
where $G, G' \subset {\mathbb R}^n$ are domains, is denoted
by $K(f)\,.$

The subject of this research is to study the following well-known
problem.

\begin{problem} \label{myprob}
\begin{enumerate}
\label{prob1}
\item
Given $a,b \in G$ and  $f\in Id(\partial G)$ with $f(a)=b,$ find a
lower bound for $K(f)$.
\item
\label{prob1-2} Given $a,b\in G,$ construct $f\in Id(\partial G)$
with $f(a)=b$ and give an upper bound for $K(f)$.
\end{enumerate}
\end{problem}

O. Teichm\"uller studied this problem in the case when $G$ is a
plane domain with $card(\overline{\mathbb R^2}\setminus G)=3$ and
solved it by proving the following theorem with a sharp bound for $K(f)$.

\begin{theorem}\label{thm1}
Let $G= \mathbb{R}^2 \setminus \{0,1\}$, $a,b\in G$. Then there
exists $f\in Id_K(\partial G)$ with $f(a)=b$ iff
$$
\log(K(f))\geqslant s_G(a,b),
$$
where $s_G(a,b)$ is the hyperbolic metric of $G$.
\end{theorem}

Motivated by a question of F.W. Gehring, J. Krzy\.z \cite[Theorem
1]{k} proved the following theorem. See also Teichm\"uller \cite{t} and Krushkal \cite[p.59]{kr}.
Write $B^n(r)=\{x\in{\mathbb{R}}^n\,:\,|x|<r\}$ and $B^n =B^n(1)$.

\begin{theorem} \label{krzyz0}
( Krzy\.z \cite[Theorem 1]{k}) \label{thm2} For $f\in Id_K(\partial
B^2)$ the sharp bounds are:
\begin{equation}
\label{krzyz1}
|f(0)|\leqslant\mu^{-1}\left(\log\frac{\sqrt{K}+1}{\sqrt{K}-1}\right)\equiv c_1
\end{equation}
where $\mu$ is the function defined in (\ref{dec_hom}) and
\begin{equation}
\label{krzyz2} \tanh\frac{\rho_{B^2}(f(z),z)}{2}\leqslant c_1
\end{equation}
for every $z\in B^2$, where $\rho_{B^2}$ is the hyperbolic metric
defined in Lemma \ref{hypmetr}.
\end{theorem}

The constant $c_1$ in (\ref{krzyz1}) is quite involved. It is hard
to see how it behaves in the crucial passage to limit $K\to1\,.$
Therefore we give an explicit bound for this constant.

\begin{lemma} \label{krzyzbound}
The constant $c_1$ in (\ref{krzyz1}) satisfies for $K>1$
$$
\frac{K-1}{K+1}<c_1<2\frac{K-1}{\sqrt{K}+1}.
$$
\end{lemma}

Later studies of this topic include the paper of G. Martin \cite{m}.
He formulated a question of the same type as Gehring did, but for
general plane domains. This question was solved in the negative, at
the same time by A. Solynin--M. Vuorinen \cite{sv} and H.
Xinzhong--N.E. Cho \cite{xc}.

Our goal here is to study the $n$-dimensional case.

For any proper domain $G\subset\mathbb R^n$ we consider the density
$\rho(x)=\frac{1}{d(x,\partial G)}, \, x \in G.$ The corresponding
metric, denoted by $k_G$ \cite{gp}, is called the quasihyperbolic
metric in $G$. Thus for $x,y\in G$,
$$
k_G(x,y)=\inf_\gamma\int_\gamma\rho\,ds,
$$
where the infimum is taken over the family of all rectifiable curves
$\gamma$ in $G$ joining $x$ to $y$.

Gehring and Palka \cite{gp} proved the following upper bound for
Problem \ref{myprob}. Presumably this bound could be improved.

\begin{theorem}
\cite[Lemma 3.1]{gp} In Problem \ref{prob1} (\ref{prob1-2}) we can
choose $K(f)\leqslant\exp(c_2 k_G(a,b) )$ where $c_2>0$ only depends
on the dimension $n$.
\end{theorem}

In the case of uniform domains with connected boundary, a lower
bound was given by the second author in \cite{vu1}, see Theorem
\ref{vu1thm} below. For the case of the unit ball this problem was
studied by G.D.~Anderson and M.~K.~Vamanamurthy \cite{av}, who found
the following counterpart for Theorem \ref{krzyz0} for dimensions
$n\ge 3.$ Note, in particular, that they use here the linear
dilatation and that an additional symmetry hypothesis is required.
They conjectured on p. 2 of \cite{av} that the result also holds
without this additional hypothesis.

\begin{theorem}
\cite{av} For $f\in Id(\partial B^n)$ with the linear dilatation
$H(f) =K$ (cf. \cite[p. 78]{v2}) we have
$$
|f(0)|\leqslant c_1,
$$
where $c_1$ is as in (\ref{krzyz1}) provided that $f$ satisfies a
certain symmetry hypothesis.
\end{theorem}

The goal of this paper is to prove the following theorem where no
extra symmetry hypotheses are required.
\begin{theorem}
\label{main_theorem} If $f\in Id_K(\partial B^n)$, then for all
$x\in B^n$
$$
\rho_{B^n}(f(x),x)\leqslant\log\frac{1-a}a,\quad
a=\varphi_{1/K,n}(1/\sqrt{2})^2,
$$
where $\rho_{B^n}$ is the hyperbolic metric
defined in Lemma \ref{hypmetr} and
$\varphi_{K,n}$ is as in (\ref{phindef}).
\end{theorem}

\begin{theorem} \label{mycor}
If $f\in Id_K(\partial B^n)$, then for all $x\in B^n, n\ge 2,$ and
$K\in[1,17]$
\begin{equation} \label{luckynumber}
  |f(x) -x| \le \frac{9}{2} (K-1) \,.
\end{equation}

For $n=2$ and $K>1$ we have
\begin{equation} \label{n=2}
|f(x)-x|\leqslant\frac{b}{2}(K-1),\quad b\leqslant 4.38.
\end{equation}
\end{theorem}

%\begin{corollary} \label{highn}
%\end{corollary}

The theory of $K$-quasiregular mappings in ${\mathbb R}^n, n \ge 3,$
with maximal dilatation $K$ close to $1\,$ has been extensively
studied by Yu. G. Reshetnyak \cite{r} under the name "stability
theory". By Liouville's theorem we expect that when  $n \ge 3$ is fixed and
$K\to 1$ the $K$-quasiregular maps "stabilize", become more and more
like M\"obius transformations, and this is the content of the deep
main results of \cite{r} such as \cite[p. 286]{r}. We have been
unable to decide whether Theorem \ref{main_theorem} follows from
Reshetnyak's stability theory in a simple way.
V. I. Semenov \cite{s} has also made significant contributions to this
theory. For the plane case, P. P. Belinskii  has found several
sharp results in \cite{be}.

Finally, it seems to be an open problem whether a new kind of stability
behavior holds: If $K>1$ is fixed, do maps in $Id_K(\partial B^n)$ approach
identity when $n \to \infty$? Our results do not answer this question. This
kind of behavior is anticipated in \cite[Open problem 9, p. 478]{avv}.

%%%%%%%%%%%%%%%%%%%%%%%%%%
%%%%%%%%%%%%%%%%%%%%%%%%%%
%%%%%%%%%%%%%%%%%%%%%%%%%%

\section{Preliminary results}

We shall follow the terminology of \cite{v2}, where for instance the
moduli of curve families are discussed. For the hyperbolic metric
$\rho_{B^n}$ of the unit ball $B^n$ our
main reference is \cite{b}. In the next lemma we give a useful estimate
(\ref{tanh}) for it. Some applications of (\ref{tanh}) were given in \cite[pp.141-142]{vu2}. Very
recently, Earle and Harris \cite{eh} have given several applications and
extended this inequality to other metrics such as the
Carath\'eodory metric.% and \cite{kl}.
\begin{lemma} \label{hypmetr}
For $x,y\in B^n$ let $t=\sqrt{(1-|x|^2)(1-|y|^2)}$. Then 
\begin{equation} \label{tanhid}
\tanh^2\frac{\rho_{B^n}(x,y)}2=\frac{|x-y|^2}{|x-y|^2+t^2}\, ,
\end{equation}
\begin{equation}
\label{tanh} |x-y|\leqslant 2\tanh\frac{\rho_{B^n}(x,y)}{4}=
\frac{2|x-y|}{\sqrt{|x-y|^2+t^2}+t}\, ,
\end{equation}
where equality holds for $x=-y$.
\end{lemma}
\begin{proof}
For (\ref{tanhid}) see \cite[p. 40]{b}, for
(\ref{tanh}) see \cite[(2.18), 2.27]{vu2}.
\end{proof}

Next, we consider a decreasing
homeomorphism $\mu:(0,1)\longrightarrow(0,\infty)$ defined by
\begin{equation}
\label{dec_hom} \mu(r)=\frac\pi 2\,\frac{{\K}(r')}{{\K}(r)}, \quad {\K}(r)=\int_0^1\frac{dx}{\sqrt{(1-x^2)(1-r^2x^2)}}\, ,
\end{equation}
where ${\K}(r)$ is Legendre's complete elliptic integral
of the first kind and $ r'=\sqrt{1-r^2},$
for all $r\in(0,1)$.

The Hersch-Pfluger distortion function is an increasing
homeomorphism $\varphi_K:(0,1)\longrightarrow(0,1)$ defined by
\begin{equation} \label{phidef}
\varphi_K(r)=\mu^{-1}(\mu(r)/K)
\end{equation}
for all $r\in(0,1)$, $K>0$. By continuity we set $\varphi_K(0)=0$,
$\varphi_K(1)=1$. From (\ref{dec_hom}) we see that
$\mu(r)\mu(r')=\left(\frac{\pi}2\right)^2$ and from this we are able
to conclude a number of properties of $\varphi_K$. For instance, by
\cite[Thm 10.5, p. 204]{avv}
\begin{equation} \label{phipyth}
\varphi_K(r)^2+\varphi_{1/K}(r')^2=1,\quad r'=\sqrt{1-r^2},
\end{equation}
holds for all $K>0$, $r\in(0,1)$.

\begin{subsec}{\bf
Proof of Lemma \ref{krzyzbound}.} By \cite[(5.27)]{avv} we have for
$y>0 $
$$  \sqrt{1- \tanh^2 y}  < \sqrt{1- \tanh^8 y}< \mu^{-1}(y)< 4 e^{-y}\, . $$
With
$$ y=
\log \frac{\sqrt{K}+1}{\sqrt{K}-1} = 2 {\rm artanh} (1/\sqrt{K})
$$
this inequality yields
$$ \frac{\sqrt{K}-1}{K+1}< c_1= \mu^{-1}(y) <  4 \frac{\sqrt{K}-1}{\sqrt{K}+1} <2 \frac{K-1}{\sqrt{K}+1}\, .
\quad \quad\square$$
\end{subsec}

\begin{subsec}{The Gr\"otzsch and Teichm\"uller rings.}
The Gr\"otzsch and Teichm\"uller ring domains $R_{G}(s), s>1,$ and
$R_T(t), t>0,$ are doubly connected domains with complementary
components $(\overline{B}^n, [se_1,\infty))$ and $([-e_1,0], [te_1,\infty)),$
respectively. Their capacities $ {\rm cap} R_G(s) $ and ${\rm cap} R_T(t)$
are often used below.
The Gr\" otzsch capacity  $\gamma_n(s) =  {\rm cap} R_G(s) $
is a decreasing homeomorphism
$\gamma_n:(1,\infty)\longrightarrow(0,\infty)$
see \cite[p.66]{vu2}, \cite[Section 8]{avv}.
The Teichm\"uller capacity $\tau_n(t)={\rm cap} R_T(t)$,
is a decreasing homeomorphism
$\tau_n: (0,\infty) \to (0,\infty)$  connected with $\gamma_n$
by the identity
\begin{equation} \label{groandteich}
\tau_n(t) = 2^{1-n} \gamma_n(\sqrt{1+t}),\, t >0.
\end{equation}
\end{subsec}

Given $E,F, G  \subset {\mathbb R}^n$
we use the notation $\Delta(E,F;G)$ for the family of all curves
that join the sets $E$ and $F$ in $G\,$ and $M(\Delta(E,F;G))$ for
its modulus, see \cite[Chapter I]{v2}. Then $\tau_n(t)= M(\Delta(E,F;{\mathbb R}^n)) $ where $E$ and $F$ are the complementary components of the Teichm\"uller
ring and a similar relation also holds for $\gamma_n(s).$

We use the standard notation
\begin{equation} \label{phindef}
\varphi_{K,n}(r)=\frac 1{\gamma_n^{-1}(K \gamma_n(1/r))}.
\end{equation}
Then $\varphi_{K,n}:(0,1)\longrightarrow(0,1)$ is an increasing homeomorphism,
see \cite[(7.44)]{vu2}. Because $\gamma_2(1/r) = 2 \pi/\mu(r)$ by
\cite[(5.56)]{vu2}, \cite{lv}, it follows that $\varphi_{K,2}(r)$ is
the same as the function $\varphi_{K}(r)$ in (\ref{phidef}).

\begin{subsec}{The key constant.} The special functions introduced
above will have a crucial role in what follows. For the sake of easy
reference we give here some well-known identities between them that
can be found in \cite{avv}. First, the function
\begin{equation}
\eta_{K,n}(t)= \tau_n^{-1}(\tau_n(t)/K) =
\frac{1-\varphi_{1/K,n}(1/\sqrt{1+t})^2}{\varphi_{1/K,n}(1/\sqrt{1+t})^2},
\, K
>0\,,
\end{equation}
defines an increasing homeomorphism $\eta_{K,n}: (0,\infty)\to (0,\infty)\,$(cf.
\cite[p.193]{avv}). The constant $(1-a)/a,
a=\varphi_{1/K,n}(1/\sqrt{2})^2,$ in Theorem \ref{main_theorem} can be
expressed as follows for $K>1$
\begin{equation}
\label{aform} (1-a)/a = \eta_{K,n}(1) =\tau_n^{-1}(\tau_n(1)/K) \, .
\end{equation}
Furthermore, by (\ref{phipyth}) %\cite[(9.17), 10.33]{avv}
\begin{equation} \label{eta1}
  \eta_{K,2}(t) = \frac{s^2}{1-s^2} ,
  \quad s= \varphi_{K,2}(\sqrt{t/(1+t)})\,
\end{equation}
and
\begin{equation} \label{eta1b}
  \eta_{K,2}(1) \in (e^{\pi(K-1)} , e^{b(K-1)})
\end{equation}
where $b= (4/\pi) {\K}(1/\sqrt{2})^2 = 4.376879...$ Note that the
constant $\lambda(K)$ in \cite[10.33]{avv} is the same as
$\eta_{K,2}(1)\,.$ In passing we remark that P. P. Belinskii gave in
\cite[Lemma 12, p. 80]{be} the inequality
$$\eta_{K,2}(1) \equiv \lambda(K) <1+ 12(K-1)$$
for $K$ close to $1\,,$ however, with an incorrect proof as pointed out in \cite[(3.10)]{aqv}.
Because this inequality is one of the key technical estimates of \cite{be}, it is fortunate
that this error was detected and a correct proof was later found (see \cite[Corollary 3.7]{aqv}).
\end{subsec}   %%%
%%%%%%%%%%%%%%%%%%%

 For the proof of Lemma \ref{stabrmk}, we record a lower        
bound for $\varphi_{1/K,n}(r)\,.$ The constant $\lambda_n \in [4, 2 e^{n-1})$ 
is the so called Gr\"otzsch ring constant, see \cite{avv}.

\begin{lemma} (\cite[7.47, 7.50]{vu2}) \label{cgqm747} For $n \ge 2, K \ge
1,$ and $0\le r \le 1$
\begin{equation}
\varphi_{1/K,n}(r) \ge \lambda_n^{1- \beta} r^{\beta}, \, \, \beta=
K^{1/(n-1)}, \label{7472}
\end{equation}
\begin{equation}  \label{7502}
\lambda_n^{1- \beta}\ge 2^{1-\beta} K^{-\beta} \ge
2^{1-K} K^{-K} \,.
\end{equation}
\end{lemma}

In the next lemma we consider two strictly increasing continuous functions
$p,q:[1,\infty) \to (0,\infty)$ such that $p(1)<q(1)$ and that the
opposite inequality $p(x_1) >q(x_1)$ holds for some $x_1>1\,.$
In the first part of the lemma we find, for the given functions,
a concrete value $\varepsilon>0$ such that $p(x)<q(x)$
for all $x \in [1,1+\varepsilon)\,.$ In the second part of the lemma
we apply an iterative method with $1+\varepsilon$ as a starting value
to find the largest number $a \in [1+\varepsilon,x_1)$ such that
$p(x)<q(x)$ for all $x\in [1,a)\,$ and show that $a>17\,.$

\begin{lemma}
\label{mn_lemma}
\begin{enumerate}
\item
For all $m,n\geqslant 1$ there is $M>1$ such that the inequality
\begin{equation}
\label{lemas_inequality}
\log(2^{mx-m+1}x^{nx}-1)\leqslant(2m\log 2+2n)(x-1)
\end{equation}
holds for $x\in[1,M]$ with equality only for $x=1$. Moreover, with
$t= (m\log 2-n)/(2n)\,,$ $M$ can be chosen as
$$
M=\sqrt{\frac{(m-1)\log 2+\log\left(1+\frac{(n+m\log
2)^2}{n}\right)}{n}+ t^2}-t.
$$
\item
Let $p(x)=\log(2^{mx-m+1}x^{nx}-1)$, $q(x)=(2m\log 2+2n)(x-1)$ and
let us use the above notation.
Let $a_0=M$ and $a_{n+1}=p^{-1}(q(a_n))$ for $n\geqslant 1$.
Then the sequence $a_n$ is increasing and bounded.
If $a=\lim_{n\rightarrow \infty}a_n$ then the inequality (\ref{lemas_inequality})
holds for $x\in[1,a]$ with equality iff $x\in\{1,a\}$.
For $m=3$ and $n=2$ we have $a>17$.
\end{enumerate}
\end{lemma}
\begin{proof}
Let
$$
u(x)=(mx-m+1)\log 2+nx\log x,\quad v(x)=
\log(e^{u(x)}-1)=\log(2^{mx-m+1}x^{nx}-1).
$$
Then we have
$$
\begin{array}{rcl}
v''(x) & = &
\displaystyle
(\log(e^{u(x)}-1))'' =
\left(
\frac{u'(x)\,e^{u(x)}}{e^{u(x)}-1}
\right)'
\vspace{1em}\\
& = &
\displaystyle
\frac{(u''(x)e^{u(x)}+(u'(x))^2e^{u(x)})(e^{u(x)}-1)-(u'(x)\,e^{u(x)})^2}{(e^{u(x)}-1)^2}
\vspace{1em}\\
& = &
\displaystyle
\frac{e^{u(x)}}{(e^{u(x)}-1)^2}\cdot
((u''(x)+(u'(x))^2)(e^{u(x)}-1)-(u'(x))^2e^{u(x)})
\vspace{1em}\\
& = &
\displaystyle
\frac{e^{u(x)}}{(e^{u(x)}-1)^2}\cdot
(u''(x)(e^{u(x)}-1)-(u'(x))^2).
\end{array}
$$
Thus
$$
v''(x)\leqslant 0\,\,\,\Leftrightarrow\,\,\,u''(x)(e^{u(x)}-1)\leqslant(u'(x))^2.
$$
Since
$$
e^{u(x)}=2^{mx-m+1}x^{nx},\quad u'(x)=n+m\log 2+n\log x,\quad u''(x)=\frac nx,
$$
we have
$$
v''(x)\leqslant 0\,\,\,\Leftrightarrow\,\,\,\frac nx(2^{mx-m+1}x^{nx}-1)\leqslant(n+m\log 2+n\log x)^2,
$$
therefore $v''(x)\leqslant 0$ is for $x\geqslant 1$ equivalent to
$$
2^{mx-m+1}x^{nx}-1\leqslant\frac xn(n+m\log 2+n\log x)^2.
$$
Let $f(x)=2^{mx-m+1}x^{nx}-1$ and $g(x)=\frac xn(n+m\log 2+n\log x)^2$.
Both functions $f$ and $g$
are increasing on $[1,+\infty)$ and $f(1)<g(1)$ because
$$
f(1)=1\leqslant n=\frac 1n\cdot n^2<\frac 1n(n+m\log 2)^2=g(1).
$$
By the continuity of $f$ we can conclude that there is $M>1$ such that
$f(M)\leqslant g(1)$. For such $M$
$$
f(x)\leqslant f(M)\leqslant g(1)\leqslant g(x),\quad x\in[1,M].
$$
This implies that $v$ is concave on $[1,M]$ and therefore
$$
v(x)\leqslant v(1)+v'(1)(x-1),\quad x\in[1,M]
$$
i.e.
$$
\log(2^{mx-m+1}x^{nx}-1)\leqslant(2m\log 2+2n)(x-1),\quad x\in[1,M].
$$
The inequality $f(x)\leqslant g(1)$ is equivalent to
\begin{equation}
\label{log_form}
(mx-m+1)\log 2+nx\log x\leqslant\log\left(1+\frac{(n+m\log 2)^2}n\right).
\end{equation}
Because
\begin{equation}
\label{linearization}
(mx-m+1)\log 2+nx\log x\leqslant(mx-m+1)\log 2+nx(x-1)
\end{equation}
the inequality (\ref{log_form}) is a consequence of the inequality
\begin{equation}
\label{quad_form}
(mx-m+1)\log 2+nx(x-1)\leqslant\log\left(1+\frac{(n+m\log 2)^2}n\right).
\end{equation}
In (\ref{linearization}) the equality sign holds only for $x=1$. Because
$$
1+\frac{(n+m\log 2)^2}n>1+\frac{n^2}{n}=1+n\geqslant 2
$$
the inequality (\ref{quad_form}) is a strict inequality for $x=1$.
By this reason, the greater root of the quadratic equation
$$
(mx-m+1)\log 2+nx(x-1)=\log\left(1+\frac{(n+m\log 2)^2}n\right)
$$
is greater than $1$. If we denote this root with $M$ the inequality
(\ref{log_form}) holds for $x\in[1,M]$ with equality only for $x=1$.
The first part of Lemma is proved.

Now we prove the second part of the inequality. Both of the functions
$p(x)$ and $q(x)$ are continuous and increasing. Consequently
$r(x)=p^{-1}(x)$ is continuous and increasing. Because
$$
p(a_1)=q(a_0)>p(a_0)
$$
using monotonicity of $p(x)$ we can conclude that $a_1>a_0$.
Now, by induction and monotonicity of $r$ we can conclude that
the sequence $a_n$ is increasing. Now for $x\in[a_n,a_{n+1})$ we have
$$
p(x)<p(a_{n+1})=q(a_n)\leqslant q(x).
$$
Therefore the inequality $p(x)<q(x)$ holds for $x\in\bigcup_{n=0}^\infty[a_n,a_{n+1})=[a_0,a)$
and using what was already proved, we see that the inequality $p(x)<q(x)$ holds for the whole
interval $1<x<a$.
For $x\geqslant 1$ we see that $mx-m+1>1$ and $x^{nx}\geqslant 1$ and consequently
$$
p(x)=\log(2^{mx-m+1}x^{nx}-1)>\log(2\,x^{nx}-1)\geqslant nx\log x.
$$
Because $p(x)>nx\log x\geqslant(n\log x)(x-1)$ the inequality $p(c)>q(c)$
holds for $c$ such that $n\log c\geqslant 2m\log 2+2n$. It is
easy to check that it is true for $c=2^{\frac{2m}{n}}e^2$.
It implies that $a$ is finite (for example $a<2^{\frac{2m}{n}}e^2$) and $a_n$ is bounded.
The relation $p(a_{n+1})=q(a_n)$ and the continuity
of both functions shows that $\lim p(a_{n+1})= p(a)=q(a) = \lim q(a_n)\,.$
The lower bound for $a$ follows because $a_{36}>17\,.$
\end{proof}

\begin{lemma} \label{stabrmk}
If $a=\varphi_{1/K,n}(1/\sqrt{2})^2$ is as in Theorem
\ref{main_theorem} then for $M>1$ and $\beta\in[1,M]$
\begin{equation} \label{nbnd}
\log\left(\frac{1-a}a\right)\le
\log(\lambda_n^{2(\beta-1)}2^{\beta}-1)\le V(n)(\beta-1)
\end{equation}
with $V(n)=(2\log(2\lambda_n^2))(2\lambda_n^2)^{M-1}$ and for
$K\in[1,17]$,
\begin{equation}
\label{remark_inequality}
\log\left(\frac{1-a}a\right)\leqslant(K-1)(4+6\log 2) <9(K-1),\quad
\end{equation}
with equality only for $K=1$. For $n=2$ and $K>1$
\begin{equation} \label{ineq2}
\log\left(\frac{1-a}a\right)=
\log\left(\frac{\varphi_{K,2}(1/\sqrt{2})^2}{\varphi_{1/K,2}(1/\sqrt{2})^2}\right)
\leqslant b(K-1)
\end{equation}
where $b= (4/\pi){\K}(1/\sqrt{2})^2\le 4.38 \,.$
\end{lemma}

\begin{proof}
For $\beta\in[1,M]$ we have by (\ref{7472})
$$
\log\left(\frac{1-a}a\right)\le
\log(\lambda_n^{2(\beta-1)}2^\beta-1) \,.
$$
Further, we have
$$
\frac{\log(\lambda_n^{2(\beta-1)}2^\beta-1)}{\beta-1}\leqslant
2\,\frac{(2\lambda_n^2)^{\beta-1}-1}{\beta-1}\leqslant
(2\log(2\lambda_n^2))(2\lambda_n^2)^{M-1}.
$$
The second inequality follows from the inequality $\log(t)\leqslant
t-1$ and the third one from Lagrange's theorem and the monotonicity of
the function $(2\log(2\lambda_n^2))(2\lambda_n^2)^{x-1}$. This
proves (\ref{nbnd}).

From (\ref{7502}) it follows that the constant $a$ satisfies the
inequality
$$
a\ge 2^{2(1-K)} K^{-2K} (1/\sqrt{2})^{2K} \,
$$
and also
$$
1/a \le 2^{3K-2} K^{2K}\,,\quad K>1.
$$
By Lemma \ref{mn_lemma} we have
$$
\log(2^{3K-2} K^{2K}-1)\leqslant(4+6\log 2)(K-1)
$$
for $K\in[1,17]$ with equality only for $K=1$. Now, from
$$
\frac{1-a}{a}<2^{3K-2} K^{2K}-1,\quad K>1\,,
$$
we conclude that
$$
\log\left( \frac{1-a}{a} \right) \leqslant(4+6\log 2)(K-1) <9(K-1)
\, .
$$

For the case $n=2$ we can apply the identity (\ref{eta1}) and 
the inequality in (\ref{eta1b}).

\end{proof}

%\begin{rmk}

%\end{rmk}

%%%%%%%%%%%%%%%%%%%%%%%%%%
%%%%%%%%%%%%%%%%%%%%%%%%%%
%%%%%%%%%%%%%%%%%%%%%%%%%%
\section{Proof of Theorem \ref{main_theorem}}

%For $E,F,G\subset\mathbb R$ we let $\Delta(E,F;G)$ stand for the
%family of all curves in $G$ joining $E$ with $F$ (cf. \cite{v2}).
%The modulus of a curve family $\Gamma$ is denoted by $M(\Gamma)$.

Lemma \ref{levu1} and Theorem \ref{thmvu1} deal with the first part
of Problem \ref{myprob}.

\begin{lemma}{\rm \cite{vu1}} \label{levu1}
Let $f\in Id_K(\partial G)$, $a,b\in G$, $f(a)=b$, and let the boundary $\partial G$ be connected.
If $x\in\partial G$ is such that $d(a)=d(a,\partial G)=|a-x|\leqslant|b-x|$,
then
$$
K(f)\geqslant\overline{d}_n\left(\log\frac{|b-x|}{|a-x|}\right)^n,
\quad
\overline d_n=\frac{c_n}{\omega_{n-1}}\,\frac{(n-1)^{n-1}}{n^n}.
$$
\end{lemma}

The following result was proved in \cite{vu1}, however, under the
condition that the points are far away from each other. The general
case follows from the original result by reducing the constant. In
\cite{vu1}, an example was given to the effect that Theorem
\ref{vu1thm} cannot be improved to the claim that $a,b \in G,
k_G(a,b)>0$ implies $K(f)>1.$
\begin{theorem} \label{thmvu1}
\label{vu1thm} {\rm\cite{vu1}} Let $f\in Id_K(\partial G)$, $a,b\in G$
with $f(a)=b$. If $G$ is a uniform domain with connected boundary $\partial G\,,$ then
$$
K(f)\geqslant d_n\,k_G(a,b)^n
$$
where $d_n$ depends only on $n$ and $G$.
 \end{theorem}

\begin{subsec}{\bf
Proof of Theorem \ref{main_theorem}.} Fix $x\in B^n$ and let $T_x$
denote a M\" obius transformation of $\overline{\mathbb R^n}$ with
$T_x(B^n)=B^n$ and $T_x(x)=0$. Define $g:\mathbb
R^n\longrightarrow\mathbb R^n$ by setting $g(z)=T_x\circ f\circ
T^{-1}_x(z)$ for $z\in B^n$ and $g(z)=z$ for $z\in\mathbb
R^n\setminus B^n$. Then $g\in Id_K(\partial B^n)$with
$g(0)=T_x(f(x))$. By the invariance of $\rho_{B^n}$ under the group
${\cal{GM}}(B^n)$ of M\"obius selfautomorphisms of $B^n$ we see that
for $x\in B^n$
\begin{equation}
\label{rho_Bn}
\rho_{B^n}(f(x),x)=\rho_{B^n}(T_x(f(x)),T_x(x))=\rho_{B^n}(g(0),0).
\end{equation}
Choose $z\in\partial B^n$ such that $g(0)\in
[0,z]=\{tz\,:\,0\leqslant t\leqslant 1\}$. Let
$E'=\{-sz\,:\,s\geqslant 1\}$, $\Gamma'=\Delta([g(0),z],E';\mathbb
R^n)$ and $\Gamma=\Delta(g^{-1}[g(0),z],g^{-1}E';\mathbb R^n)$. Observe
that $E'= g^{-1}E'$ because $g\in Id_K(\partial B^n)\,.$

The spherical symmetrization with center at $0$ yields by \cite[Thm 8.44]{avv}
$$
M(\Gamma)\geqslant\tau_n(1)\quad(=2^{1-n}\gamma_n(\sqrt 2))
$$
because $g(x)=x$ for $x\in\mathbb R^n\setminus B^n$.
%Here $\tau_n:(0,\infty)\longrightarrow(0,\infty)$
%is a decreasing homeomorphism, the Teichm\" uller capacity, defined by
%$$
%\gamma_n(s)=2^{n-1}\tau_n(s^2-1),\quad s>1.
%$$
Next, we see by the choice of $\Gamma'$ that
$$
M(\Gamma')=\tau_n\left(\frac{1+|g(0)|}{1-|g(0)|}\right).
$$
By $K$-quasiconformality we have $M(\Gamma)\leqslant K\,M(\Gamma')$ implying
\begin{equation}
\label{main_inequality}
\exp(\rho_{B^n}(0,g(0)))=\frac{1+|g(0)|}{1-|g(0)|}\leqslant\tau_n^{-1}(\tau_n(1)/K)=\frac{1-a}a.
\end{equation}
The last equality follows from (\ref{aform}). %the functional
%identity
%$$
%\tau_n^{-1}(K\,\tau_n(s)/K)=\frac{1}{\varphi_{1/K,n}(1/\sqrt{1+s})^2}-1,\quad K,s>0,
%$$
%see \cite[9.16]{avv}.
Finally, (\ref{rho_Bn}) and (\ref{main_inequality}) complete the proof. $ \hfill \square$
\end{subsec}

%\begin{pfwithnr}
%\end{pfwithnr} { {\vskip -0.8cm \hskip 2mm} \bf Proof of Theorem \ref{mycor}.}
%\begin{proof}
\begin{subsec}{\bf
Proof of Theorem \ref{mycor}.} We have
$$
\begin{array}{rcl}
|f(x)-x| & \leqslant & \displaystyle
2\tanh\left(\frac{\rho_{B^n}(f(x),x)}{4}\right)
\leqslant
2\tanh\left(\frac{\log\left(\frac{1-a}{a}\right)}{4}\right)
\vspace{1em}\\
& \leqslant & \displaystyle
2\tanh\left(\frac{(K-1)(4+6\log 2)}{4}\right)
\vspace{1em}\\
& \leqslant & \displaystyle
(K-1)(2+3\log 2)\leqslant\frac 92(K-1).
\end{array}
$$
The first inequality follows from (\ref{tanh}), the second one from Theorem \ref{main_theorem},
the third one from Lemma \ref{stabrmk} and the fourth one from the inequality
$\tanh(t)\leqslant t$ for $t\geqslant 0$.

%{\bf For the proof of (\ref{nbig}).....}

For $n=2$ we use the same first two steps and the planar case of 
Lemma \ref{stabrmk} to derive the inequality
$$
|f(x)-x|\leqslant\frac{b}{2}(K-1). \quad \quad \hfill \square
$$
\end{subsec}

A lower bound corresponding to the upper bound in
(\ref{luckynumber}) is given in the next lemma.

\begin{lemma} For $f \in Id(\partial G)$ let
$$
\delta(f)  \equiv \sup \{|f(z)-z| : z \in G\}\,.
$$
Then for $f \in Id_K(\partial B^n), K>1,  \alpha=K^{1/(1-n)}$
\begin{equation} \label{**}
\delta(f) \ge (1-\alpha) \alpha^{\alpha/(1-\alpha)}
>\frac{1}{e}(1-\alpha).
\end{equation}
\end{lemma}

\begin{proof}
 The radial stretching $f: B^n \to B^n, n \ge 2, $ defined by
$f(z)=|z|^{\alpha-1}\,z, z \in B^n,$ ($0<\alpha<1$) is $K$-qc with
$\alpha=K^{1/(1-n)}$ \cite[p. 49]{v2} and $f \in Id_K(\partial
B^n)\,.$ Now we have
$$
|f(z)-z|=||z|^{\alpha-1}z-z|=|r^\alpha-r|,\quad |z|=r.
$$
Further, we see that
$$
\delta(f)=\sup_{0<r<1}(r^{\alpha}-r),
$$
where the supremum is attained for
$r=r_\alpha=\left(\frac{1}{\alpha}\right)^{\frac{1}{\alpha-1}}$, so
$$
\delta(f) =({1}-{\alpha}) \alpha^{\alpha/(1- \alpha)}\,.
$$
A crude, but simple, estimate is
$$%\begin{equation}
\delta(f)\ge (1/e)^{\alpha} - (1/e)
=\frac{1}{e}\left(\frac{1}{e^{\alpha
-1}}-1\right)=\frac{1}{e}\left(e^{1-\alpha}-1\right)
\geqslant\frac{1}{e}(1-\alpha)\, .
$$ %\end{equation}
\end{proof}

\begin{theorem} \label{etathm}
\label{bounds} Let $f:\overline{\mathbb
R^n}\longrightarrow\overline{\mathbb R^n}$ be a $K$-qc homeomorphism
with $f(\infty)=\infty$ and $B^n(m)\subset f(B^n) \subset B^n(M) $ where $0< m \le 1 \le M$. Then
$$
\eta_{1/K,n}\left(\frac{1+|x|}{1-|x|}\right)\leqslant\frac{M+|f(x)|}{m-|f(x)|}
$$
and
$$
\frac{m+|f(x)|}{M-|f(x)|}
\leqslant
\eta_{K,n}\left(\frac{1+|x|}{1-|x|}\right)
$$
for all $x\in B^n$ where $\eta_{K,n}(t)=\tau_n^{-1}(\tau_n(t)/K)$.

In particular, if $m=1=M$, then we have
$$
 \eta_{1/K,n}\left(\frac{1+|x|}{1-|x|}\right)\leqslant\frac{1+|f(x)|}{1-|f(x)|}
\leqslant
\eta_{K,n}\left(\frac{1+|x|}{1-|x|}\right)\, .$$
\end{theorem}

\begin{proof}
The proof is similar to the proof of Theorem \ref{main_theorem}. Fix
$x\in B^n$ and choose $z'\in\partial f(B^n)$ such that $f(x)\in[0,z']$ and $[f(x),z') \subset f(B^n)\,$
and fix $z" \in \partial f(B^n)$ such that $z',0,z"$ are on the same line, $0 \in[z', z"],$ and
$ \{-sz"\,:\,s\geqslant 1\} \subset {\mathbb R}^n \setminus f(B^n)\,.$
Let $\Gamma'=\Delta([f(x),z'],E';\mathbb{R}^n)$,
$E'=\{-sz"\,:\,s\geqslant 1\}$ and
$\Gamma=\Delta(f^{-1}[f(x),z'],f^{-1}E';\mathbb R^n)$. Then
$$
%M(\Gamma')=\tau_n\left(\frac{1+|f(x)|}{1-|f(x)|}\right)
M(\Gamma')\le\tau_n\left(\frac{m+|f(x)|}{M-|f(x)|}\right)
$$
while applying a spherical symmetrization with center at the origin
gives
$$
M(\Gamma)\geqslant\tau_n\left(\frac{1+|x|}{1-|x|}\right)
$$
because $f^{-1}E'$ connects $\partial B^n$ and $\infty$.
Then the inequality $M(\Gamma)\leqslant K\,M(\Gamma')$
yields
$$
 \tau_n\left(\frac{1+|x|}{1-|x|}\right) \le K
 \tau_n\left(\frac{m+|f(x)|}{M-|f(x)|}\right),
$$
$$
\tau_n^{-1}( \frac{1}{K}\tau_n\left(\frac{1+|x|}{1-|x|}\right))
\ge\frac{m+|f(x)|}{M-|f(x)|}
$$

\begin{equation}
\label{eta}
\frac{m+|f(x)|}{M-|f(x)|}\leqslant\eta_{K,n}\left(\frac{1+|x|}{1-|x|}\right).
\end{equation}
The lower bound follows if we apply a similar argument to $f^{-1}$ and the lower bound
$$
M(\Gamma')\ge\tau_n\left(\frac{M+|f(x)|}{m-|f(x)|}\right) \,.
$$
\end{proof}

\begin{subsec}{\bf Remark.} \label{mynewobs}
Putting $x=0, m=1=M$ in (\ref{eta}) we obtain by (\ref{aform})
for  a $K$-qc homeomorphism $f:\overline{\mathbb
R^n}\longrightarrow\overline{\mathbb R^n}$
with $f(\infty)=\infty$ and $ f(B^n) = B^n $  that
$$   |f(0)| \le 1- 2a\,, a= \varphi_{1/K,n}(1/\sqrt{2})^2\,.$$
Further, if we use the lower bound (\ref{7502}) from Lemma
\ref{cgqm747} we obtain
$$   |f(0)| \le 1- 2^{1- \beta} 4^{1-K} K^{-2K}\,. $$
In the special case when $n=2$ we have
$$|f(0)| \le 1- 2^{3(1- K)}  K^{-2K}\le (2+ 3 \log 2)(K-1)\,.  $$
Note that this last inequality does not suppose that $f \in Id_K(\partial B^n)\,,$
only the hypotheses of Theorem \ref{etathm} are needed.
\end{subsec}

\begin{corollary} \label{corollary_eta} Let $n=2$ in addition to
the hypotheses of Theorem \ref{bounds}. Then
\begin{equation}
\label{corollary_equation}
\eta_{K,2}(t)=\frac{u^2}{1-u^2}=\frac{u^2}{v^2},
\end{equation}
where $u=\varphi_{K,2}\left(\sqrt{\frac{t}{1+t}}\right)$, $v=\varphi_{1/K,2}\left(\frac{1}{\sqrt{1+t}}\right)$
and
\begin{equation}
\label{corollary_inequality} |f(x)|\leqslant
2\,\varphi_{K,2}\left(\sqrt{\frac{1+|x|}{2}}\right)^2-1
\end{equation}
for all $x\in B^2$.
\end{corollary}
\begin{proof}
The identity (\ref{corollary_equation}) holds by (\ref{eta1}). Next
Theorem \ref{bounds} together with (\ref{corollary_equation}) yields
$$
\frac{1+|f(x)|}{1-|f(x)|}\leqslant\frac{w^2}{1-w^2}
$$
where $w=\varphi_{K,2}\left(\sqrt{\frac{1+|x|}2}\right)$. Solving
this for $|f(x)|$ yields (\ref{corollary_inequality}).
\end{proof}

\begin{rmk} {\rm
By the $K$-quasiconformal Schwarz lemma if $f:B^2\longrightarrow
B^2$ is $K$-quasiconformal with $f(0)=0$ then
$|f(z)|\leqslant\varphi_{K,2}(|z|)$, for all $z\in B^2$, where the
sharp bound is attained for a map with $f(B^2)=B^2$ (\cite{lv}). Note that in
Corollary {\ref{corollary_eta}} the condition $f(0)=0$ is not
required. We conclude that
\begin{equation}
\label{phi_K2} \varphi_{K,2}(r)\leqslant
2\,\varphi_{K,2}(\sqrt{\frac{1+r}{2}})^2-1.
\end{equation}
Writing $A(r,s)=\sqrt{\frac{r+s}{2}}$ {(\ref{phi_K2})} says that if
$t=1, r\in (0,1)$ then
$$
A(\varphi_{K,2}(t),\varphi_{K,2}(r))\leqslant\varphi_{K,2}(A(t,r)).
$$
It seems natural to expect that this inequality holds for all $t,
r\in (0,1)\,.$ }
\end{rmk}

\bigskip

{\bf Acknowledgement.} The research of the second author was supported
by the project "Quasiconformal Maps'' nr 209539  of Matti Vuorinen
funded by the Academy of Finland.
%%%%%%%%%%%%%%%%%%%%%%%%%%
%%%%%%%%%%%%%%%%%%%%%%%%%%
%%%%%%%%%%%%%%%%%%%%%%%%%%

\bigskip

V. Manojlovi\'c,
Department of Mathematics,
Jove Ilica 154, 11000 Beograd, SerbiaFON,
 email: {\tt  vesnak@fon.bg.ac.yu}

M. Vuorinen, Department of Mathematics,
FIN-20014 University of Turku, Finland.
email: {\tt vuorinen@utu.fi}

\end{document}